\documentclass{interact}
\usepackage[dvipsnames]{xcolor}

\usepackage{epstopdf}
\usepackage[caption=false]{subfig}

\usepackage[numbers,sort&compress]{natbib}

\usepackage[utf8]{inputenc}
\usepackage{lmodern}
\usepackage{amsmath} \usepackage{amssymb}
\usepackage{amsthm}
\newtheorem{theorem}{Theorem}[section]
\newtheorem{proposition}[theorem]{Proposition}
\newtheorem{corollary}[theorem]{Corollary}
\theoremstyle{definition}
\newtheorem{definition}[theorem]{Definition}
\newtheorem{remark}[theorem]{Remark}
\newtheorem*{init}{Initialization}
\newtheorem*{mainloop}{Update Step}
\newtheorem*{vs}{Vertex Selection}
\newtheorem*{vs_efficient}{Efficient Implementation of the Vertex Selection}
\newtheorem{example}{Example}[section]

\newtheorem*{Assumptions}{Assumptions}

\DeclareMathOperator{\interior}{int}
\DeclareMathOperator{\cl}{cl}
\DeclareMathOperator{\conv}{conv}
\DeclareMathOperator{\cone}{cone}

\DeclareMathOperator{\relint}{ri}
\DeclareMathOperator{\dom}{dom}
\DeclareMathOperator{\rank}{rank}

\DeclareMathOperator{\vertices}{vert}

\newcommand{\vrtxslct}{vertex~selection}
\newcommand{\I}{\mathcal{I}}
\renewcommand{\O}{\mathcal{O}}
\newcommand{\isassigned}{\leftarrow}
\newcommand{\href}[2]{{#2}}
\usepackage{color}
\usepackage{enumitem}
\usepackage[misc]{ifsym}
\usepackage{ifthen}
\usepackage{longtable}
\usepackage{booktabs}
\usepackage{makecell}
\usepackage[linesnumbered,algoruled,vlined]{algorithm2e}
\newcommand{\cell}[2]{\begin{tabular}{#1} #2 \end{tabular}} 
\usepackage{mathtools}
\newcommand{\R}{\mathbb{R}}
\renewcommand{\P}{\mathcal{P}}

\newcommand{\T}{\mathsf{T}}
\newcommand{\haus}[2]{d_{\mathsf{H}}(#1,#2)}
\newcommand{\decoRule}{\rule{.8\textwidth}{.4pt}}
\newcommand{\norm}[1]{\left\lVert {#1} \right\rVert}
\usepackage{pifont}
\newcommand{\cmark}{\textcolor{OliveGreen}{\ding{51}}}
\newcommand{\xmark}{\textcolor{Mahogany}{\ding{55}}}

\begin{document}


\title{A Benson-Type Algorithm for Bounded Convex Vector Optimization Problems with Vertex Selection}

\author{ \name{Daniel Dörfler\textsuperscript{a}\thanks{CONTACT
D. Dörfler. \Letter \ \href{}{daniel.doerfler@uni-jena.de}, \protect
\includegraphics[keepaspectratio, width=0.8em]{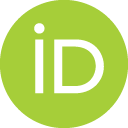}
\href{}{https://orcid.org/0000-0002-9503-3619}}, Andreas
Löhne\textsuperscript{a}, Christopher Schneider\textsuperscript{b} and
Benjamin Weißing\textsuperscript{a}} \affil{\textsuperscript{a}
Friedrich Schiller University Jena, Germany; \textsuperscript{b} Ernst
Abbe University of Applied Sciences Jena, Germany} }

\maketitle

\begin{abstract}
  We present an algorithm for approximately solving bounded convex
  vector optimization problems. The algorithm provides both an outer
  and an inner polyhedral approximation of the upper image. It is a
  modification of the primal algorithm presented by Löhne, Rudloff,
  and Ulus in 2014. There, vertices of an already known outer
  approximation are successively cut off to improve the approximation
  error. We propose a new and efficient selection rule for deciding
  which vertex to cut off. Numerical examples are provided which
  illustrate that this method may solve fewer scalar problems overall
  and therefore may be faster while achieving the same approximation
  quality.
\end{abstract}


\begin{amscode}
  90C29, 90C25, 90-08, 90C59
\end{amscode}

\section{Introduction}
There exists a variety of methods for (approximately) solving vector
optimization problems. One of the most studied and best understood
class is vector linear programming (VLP). There are numerous
algorithms for VLP as surveyed by \citeauthor{Ehr05} in
\cite{Ehr05}. These include the multiple objective simplex
method, where the set of all efficient solutions is computed in the
preimage space (or variable space) of the problem. In
\cite{Ben98}, an article from \citeyear{Ben98},
\citeauthor{Ben98} proposes an approximation algorithm that
computes the set of all efficient values by constructing a sequence of
outer approximation polyhedra in the image space (or objective
space). This is motivated by the idea that a decision maker tends to
choose a solution based on objective function values rather than
variable values, many efficient solutions may be mapped to the same
efficient point and the dimension of the image space is typically much
smaller than that of the preimage space. Although algorithms taking
these considerations into account are frequently named after Benson,
some of his ideas can be traced back to earlier works in different
areas of research. In \cite{Dau86} from \citeyear{Dau86},
\citeauthor{Dau86} analyzes the image space in VLP and observes that
the number of objectives is typically smaller than the number of
variables. More than 15 years prior to \citeauthor{Ben98}'s
article the idea of approximation polyhedra had been used in global
optimization, compare \cite{Vei67,Tuy83,Thi83}. The ideas applied in
these works can in turn be dated back to \citeauthor{Che59}
\cite{Che59} from \citeyear{Che59} and \citeauthor{Kel60} \cite{Kel60}
from \citeyear{Kel60}, who use cutting plane methods to solve convex
programs. In \cite{Las90} from \citeyear{Las90} \citeauthor{Las90}
propose an algorithm for computing projections of polyhedra by
successive refinements of approximations.  Their approach can be
viewed as a dual variant of the outer polyhedral approximation
algorithm (also compare \cite{Ehr12}). In \cite{Kam92}
\citeauthor{Kam92} formulates a framework for the approximation of
convex bodies by polyhedra. In this article from \citeyear{Kam92}, the
same ideas as in \citeauthor{Ben98}'s algorithm are used
already. An adequate solution concept for VLP based on the image space
approach is presented in \cite{Hey11,Loe11}. Various
modifications of \citeauthor{Ben98}'s algorithm for VLP have
since been developed, see
e.g. \cite{Sha08_a,Sha08_b,Loe11,Ehr12}. Improvements
of these methods where fewer LPs have to be solved per iteration are
presented in \cite{Ham14,Csi16}.

Naturally, there has been effort to extend Benson's algorithm from VLP
to the more general class of vector convex programming (VCP) or convex
vector optimization problems (CVOPs). Therefore solution concepts have
been refined to adapt to approximate solutions, see the survey article
by \citeauthor{Ruz05} \cite{Ruz05}. However, a finite
description of an approximate solution in terms of points and
directions may not be possible for an unbounded problem, see
\cite{Ulu18}. For example, the epigraph of a parabola can not be
approximated by a polyhedron, i.e. their Hausdorff distance is always
infinite. In \citeyear{Ehr11},
\citeauthor{Ehr11} \cite{Ehr11}
propose an approximation algorithm for bounded VCP motivated by
\citeauthor{Ben98}'s arguments for VLP. In \cite{Loe14} the
authors develop an algorithm that generalizes and simplifies this
approach. In particular, their method allows the use of (1) not
necessarily differentiable objective and constraint functions, (2)
more general ordering cones, and is simpler in the sense that (3) only
one convex program has to be solved in every iteration throughout the
algorithm. Moreover, a dual variant of the algorithm is provided.

In this paper we present a modification of the primal algorithm from
\cite{Loe14}. It computes sequences of polyhedral inner and outer
approximations of the upper image. In every iteration one vertex of
the outer approximation is cut off to refine the approximation
error. This requires solving one scalarization that is a convex
program in which the vertex is passed as a parameter. In \cite{Loe14}
this vertex is chosen arbitrarily. Here, we choose this vertex
according to a specific heuristic which takes into consideration the
Hausdorff distance between the current inner and outer
approximations. This rule requires to solve convex quadratic
subproblems. They differ from the scalarizations in the sense that the
variables come from the (typically lower dimensional) image space of
the vector program rather than the preimage space. Therefore, solving
the subproblems is typically cheaper than solving a
scalarization. Moreover, we show that not all subproblems have to be
solved. Instead, optimality of solutions known from prior iterations
can be verified by checking a single inequality. One advantage of this
selection rule is that the approximation error is known at every time
throughout the algorithm at no additional cost, whereas in
\cite{Loe14} it is only known either at termination or after solving a
number of scalarizations whose quantity typically increases with every
iteration. We provide three examples comparing the method presented
here with the original algorithm and illustrate its advantages. The
first one is an academic example where the modification's performance
is not affected by a certain problem parameter, whereas the original
algorithm's runtime increases with the value of the parameter. In the
second example we apply the method to the problem of regularization
parameter tracking in machine learning. This has first been done by
the authors of \cite{Gie19}. The last example concerns a real world
problem from mechanical engineering. We use the algorithm presented
here to analyze a truss design and find optimal distributions of loads
among the trusses' beam connections. In all examples fewer
scalarizations need to be solved with the modification. This leads to
(1) a decrease in runtime and (2) a smaller solution set while
achieving the same approximation quality, which is preferred by
decision makers as the amount of alternatives to choose from is less
overwhelming.

This paper is organized as follows. In Section \ref{sec:2} the
necessary notation is provided along with basic concepts. Section
\ref{sec:3} is dedicated to the problem formulation and the
theoretical background of VCP. A solution concept and scalarization
techniques are presented. The vertex selection along with the modified
version of the primal algorithm from \citep{Loe14} are presented
in Section \ref{sec:4}, correctness is proven, and a method for an
efficient implementation is discussed. Numerical examples are provided
in Section \ref{sec:5}.

\section{Preliminaries}\label{sec:2}
Given a set $A \subseteq \R^q$, we denote by $\cl A$, $\interior A$,
$\relint A$, $\conv A$, $\cone A$ the closure, interior, relative
interior, convex hull, and conic hull of $A$, respectively. We recall
that every polyhedral set $A$ can be written as the intersection of
finitely many closed halfspaces, i.e.
\begin{equation}\label{eq:hrep}
  A=\bigcap_{i=1}^\ell\{x \in \R^q \mid w_i^\T x \geqslant \gamma_i\}
\end{equation}
for $\ell \in \mathbb{N}$, $w_i \in \R^q$, $\gamma_i \in \R$ for all
$i=1,\dots,\ell$. A set ${\{(w_i,\gamma_i) \mid i=1,\dots,\ell\}}$ of
parameters fulfilling \eqref{eq:hrep} is called
\emph{$H$-representation of $A$}. Equivalently, $A$ can be expressed
as
\begin{equation}\label{eq:2}
  A=\conv\{v^1,\dots,v^s\} + \cone\{d^1,\dots,d^r\}
\end{equation}
for $s \in \mathbb{N}$, $r \in \mathbb{N}_0$, $v^i \in \R^q$, and
$d^i \in \R^q\setminus\{0\}$, that is the Minkowski sum of the convex
hull of finitely many points and the conic hull of finitely many
directions. We set $\cone \emptyset = \{0\}$. The data
$(\{v^1,\dots,v^s\},\{d^1,\dots,d^r\})$ from Equation \eqref{eq:2} are
called a~\emph{\mbox{$V$-representation} of $A$}. When expressing $A$
by $V$-representation, we will interchangeably write
$A=\conv V + \cone D$ for matrices $V \in \R^{q \times s}$ and
$D \in \R^{q \times r}$ where the columns of~$V$ and $D$ are the $v^i$
and $d^i$ in \eqref{eq:2}, respectively. A pointed convex cone
$C \subseteq \R^q$ induces a partial order $\leqslant_C$ on $\R^q$ by
$$x\leqslant_Cy \text{ if and only if } y-x \in C.$$ The nonnegative orthant of $\R^q$
is denoted by $\R^q_+$ and induces the natural (or component-wise)
order on $\R^q$ which we denote by $\leqslant$ rather than by
$\leqslant_{\R^q_+}$. The dual cone $C^+$ of~$C$ is the set
$C^+:=\{y \in \R^q \mid \forall x \in C \colon y^\T x \geqslant
0\}$. We call $C$ \emph{polyhedral} if there is a
matrix~${D \in \R^{q \times r}}$, such that
${C = \cone D := \{D\mu \mid \mu \geqslant 0\}}$. We summarize some
important facts about the set~$C=\cone D$ \cite[see][]{Kai11, Gre84}:
\begin{enumerate}[label=(\arabic*)]
\item There is a matrix $Z \in \R^{q \times \ell}$ such that $C = \{x
  \in \R^q \mid Z^\T x \geqslant 0\}$. In particular, $C^+ = \cone Z$.\label{pc:1}
\item $C = (C^+)^+$. \label{pc:2}
\item $C$ is pointed if and only if $\rank Z = q$.
\item $\interior C = \{x \in \R^q \mid Z^\T x > 0\}$.
\end{enumerate}
From \ref{pc:1} we obtain that $x\leqslant_Cy$ if and only if
$Z^\T x \leqslant Z^\T y$ for $x,y \in \R^q$. For a
set~${A \subseteq \R^q}$ and a pointed convex cone $C \subseteq \R^q$
an element $x \in A$ is called $C$\textit{-minimal} if
$(\{x\}-C\setminus\{0\}) \cap A = \emptyset$ and, if
$\interior C \neq \emptyset$, $x \in A$ is called \textit{weakly}
$C$\textit{-minimal} if~${(\{x\}-\interior C) \cap A =
  \emptyset}$. Given nonempty sets $A,B \subseteq \R^q$ we denote by~
$\haus{A}{B}$ the \textit{Hausdorff distance} between $A$ and $B$
which is defined as
\begin{equation}
  \haus{A}{B} := \max\left\lbrace\adjustlimits\sup_{a \in A}\inf_{b
    \in B} \norm{a-b}, \adjustlimits\sup_{b \in B}
  \inf_{a \in A} \norm{a-b} \right\rbrace,
\end{equation}
where $\norm{\cdot}$ denotes the euclidean norm in $\R^q$. It is well
known that $\haus{\cdot}{\cdot}$ defines a metric on the space of
nonempty compact subsets of $\R^q$. The Hausdorff distance between
arbitrary sets may be infinite. It holds true, however, that for
nonempty compact sets~${A,B \subseteq \R^q}$ and convex cones
$C^1,C^2 \subseteq \R^q$ the value of $\haus{A+C^1}{B+C^2}$ is finite
if and only if $\cl C^1 = \cl C^2$. Moreover, if $A$ and $B$ are
polyhedra with the same pointed recession cone one has
\begin{equation}\label{eq:4}
  \haus{A}{B} = \max \left\lbrace \adjustlimits\max_{a \in \vertices A} \min_{b \in
    B} \norm{a-b}, \adjustlimits\max_{b \in \vertices B} \min_{a
    \in A} \norm{a-b} \right\rbrace,
\end{equation}
where $\vertices A$ and $\vertices B$ denote the set of vertices of
$A$ and $B$, respectively. For proofs of the above statements we refer
the reader to \cite{Bat86}. The domain of an extended real-valued
function $g \colon \R^q \to \R \cup \{\infty\}$ is written as
$\dom g$. Given a function $f \colon \R^n \to \R^q$ and a pointed
convex cone $C \subseteq \R^q$, $f$ is called $C$\textit{-convex} if
for $x,y \in \R^n$ and $\lambda \in [0,1]$ it holds
\begin{equation}
  f(\lambda x + (1-\lambda) y) \leqslant_C \lambda f(x) + (1-\lambda) f(y).
\end{equation}

\section{Vector Convex Programs}\label{sec:3}
A vector convex program (VCP) is given as
\begin{equation}\label{P}
  \min \; F(x) \; \text{w.r.t} \; \leqslant_C \; \text{s.t.} \; g(x)
  \leqslant 0,\tag{P}
\end{equation}
where $F \colon X \to \R^q$ is a $C$-convex function, in particular,
$X \subseteq \R^n$ is a convex set and~${C \subseteq \R^q}$ is a
pointed convex cone. The constraint function is given as~
${g=(g_1,\dots,g_m)^\T}$, where for every $i=1,\dots,m$ the component~
${g_i\colon\R^n \to \R\cup\{\infty\}}$ is a convex function. We set
$\dom g := \bigcap_{i=1}^m \dom g_i$. Hence, $g$ is an $\R^m_+$-convex
(component-wise convex) function. The feasible set of (\ref{P}) is
denoted by $S$, i.e.~${S=\{x \in X \mid g(x)\leqslant 0\}}$ and its
image under $F$ by $F[S]$. Throughout this article we make the
following additional assumptions about \eqref{P}:
\begin{Assumptions}~
  \begin{enumerate}[label=(A\arabic*)]
  \item The objective function $F\colon X \to \R^q$ is continuous.\label{ass:A1}
  \item The constraint functions $g_i\colon \R^n \to \R\cup\{\infty\}$,
    $i=1,\dots,m$, are proper, lower semi-continuous, and their domains
    are relatively open.\label{ass:A2}
  \item $\bigcap_{i=1}^m \relint \{x \in X\mid g_i(x)\leqslant 0\} \neq
    \emptyset$\label{ass:A3}
  \item The feasible region $S$ of \eqref{P} is bounded.\label{ass:A4}
  \item The cone $C$ has nonempty interior and is given as ${C = \{x
      \in \R^q \mid Z^\T x \geqslant 0\}}$.\label{ass:A5}
  \end{enumerate}
\end{Assumptions}
\begin{definition}
  Given a VCP (\ref{P}) the set
  \begin{equation}
    \P := \cl(F[S]+C)
  \end{equation}
  is called the \textit{upper image of} (\ref{P}). We say that
  (\ref{P}) is \textit{bounded} if there exists some $y \in \R^q$ such
  that $\P \subseteq \{y\}+C$.
\end{definition}
Clearly, $\P$ is a closed and convex set.
\begin{definition}\label{def:sol}
  A point $x \in S$ is called a (\textit{weak}) \textit{minimizer for}
  (\ref{P}) if $F(x)$ is a (weakly)~$C$-minimal element of $F[S]$. A
  nonempty subset $\mathcal{X} \subseteq S$ is called an
  \textit{infimizer of} (\ref{P}) if $\cl\conv
  (F[\mathcal{X}]+C)=\P$. An infimizer $\mathcal{X} \subseteq S$ is
  called a (\textit{weak}) \textit{solution of} (\ref{P}) if it
  consists of (weak) minimizers only.
\end{definition}
This type of solution concept is introduced and studied in
\cite{Hey11} where Definition \ref{def:sol} is called a \emph{mild
  solution}.  It has been adapted to the case of VLP in
\cite{Loe11} where one is interested in finite solutions
consisting of minimal points and directions. The solution concept is
extended to finite approximate solutions for bounded VCPs in
\citep{Loe14}.
\begin{definition}\label{def:3}
  A nonempty finite subset $\mathcal{X} \subseteq S$ is called an
  $\varepsilon$\textit{-infimizer} for a bounded problem (\ref{P}) if
  \begin{equation}\label{eq:6}
    \haus{\conv F[\mathcal{X}]+C}{\P} \leqslant \varepsilon.
  \end{equation}
  A finite $\varepsilon$-infimizer $\mathcal{X} \subseteq S$ is called
  a (\textit{weak}) $\varepsilon$\textit{-solution of} (\ref{P}) if it
  consists of (weak) minimizers only.
\end{definition}
An illustration of the definition can be seen in Figure \ref{fig:1}.
\begin{figure}
\centering
\includegraphics{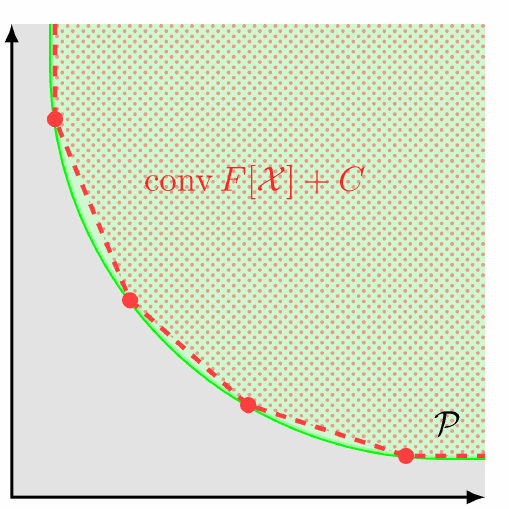}
\decoRule
\caption{\label{fig:1}%
  Illustration of an $\varepsilon$-solution ${\mathcal X}\subseteq S$
  (see Definition~\ref{def:3}).  The four points of $F[\mathcal{X}]$
  are $C$-minimal in~$\P$, hence ${\mathcal X}$ is a set of
  minimizers.  The Hausdorff~distance between ${\conv F[\mathcal{X}]
    +C}$ and $\P$ is $\varepsilon$.%
}
\end{figure}
\begin{remark}
  The original definition of an $\varepsilon$-infimizer given in
  \cite{Loe14} is a different one. There, condition \eqref{eq:6}
  is replaced by
  \begin{equation}\label{eq:7}
    \conv F[\mathcal{X}]+C-\varepsilon\{c\} \supseteq \P
  \end{equation}
  for some fixed direction $c \in \interior C$. Clearly, if
  \eqref{eq:7} holds one has \begin{equation} \haus{\conv
      F[\mathcal{X}]+C}{\P} \leqslant \varepsilon\left\lVert c
    \right\rVert.
  \end{equation}
  Since $C$ is a cone we could choose $c \in \interior C$ such that
  $\left\lVert c \right\rVert =1$. Then \eqref{eq:7} implies that~
  $\mathcal{X}$ is an $\varepsilon$-infimizer in the sense of
  Definition \ref{def:3}. The converse is also true up to a constant:
  \begin{proposition}
    Let $\mathcal{X} \subseteq S$ be an $\varepsilon$-infimizer for a
    bounded problem \eqref{P} according to Definition \ref{def:3}
    and let $C$ be closed. Then for every
    $c \in \interior{C}$ with $\norm{c} = 1$ and every~${k \geqslant
      (\min \{w^\T c \mid w \in C^+, \norm{w} = 1\})^{-1}}$ it holds that
      \[
        \conv{F[\mathcal{X}]} + C - k\varepsilon\{c\} \supseteq \P.
      \]      
  \end{proposition}
  \begin{proof}
    Since $\conv F[\mathcal{X}]+C$ is non-empty, closed and convex, it
    can be written as an intersection of closed halfspaces
    \cite[see][Theorem 18.8]{Roc70}, i.e.
    \[
      \conv F[\mathcal{X}]+C = \bigcap_{i \in I} \{y \in \R^q \mid
      w_i^\T y \geqslant \gamma_i\}
    \]
    for $w_i \in \R^q\setminus\{0\}$, $\norm{w_i}=1$,
    $\gamma_i \in \R$, and some index set $I$. Because the recession
    cone of $\conv F[\mathcal{X}]+C$ is $C$, we have $w_i \in C^+$ for
    all $i \in I$. Therefore~${w_i^\T c > 0}$ for all $i \in I$
    \cite[][p. 64]{Boy04} and
    ${k_p := \inf\{t \geqslant 1 \mid p+t\varepsilon c \in \conv
      F[\mathcal{X}]+C\}}$ exists for all~${p \in \P}$. It remains to
    show that
    ${(\min \{w^\T c \mid w \in C^+, \norm{w} = 1\})^{-1} \geqslant
      \sup\{k_p \mid p \in \P\}}$. Therefore, let $p \in \P$ such that
    $k_p > 1$. If no such $p$ exists we are done,
    because~${(\min \{w^\T c \mid w \in C^+, \norm{w} = 1\})^{-1} \in
      [1,\infty)}$. Otherwise there exists $j \in I$ such that
    $w_j^\T (p+k_p\varepsilon c) = \gamma_j$.  Denote by $d$ the
    euclidean distance from $p$ to the hyperplane defined by
    $(w_j,\gamma_j)$, i.e.\ $d = \gamma_j - w_j^\T p$. Then we obtain
    $k_p=d(\varepsilon w_j^\T c)^{-1}$.  Next, observe that
    $d \leqslant \varepsilon$: Because
    $\haus{\conv F[\mathcal{X}]+C}{\P} \leqslant \varepsilon$, there
    exists a direction $u \in \R^q$ with
    $\norm{u} \leqslant \varepsilon$ such that
    $p+u \in \conv F[\mathcal{X}]+C$. Assuming~${d > \varepsilon}$
    yields~${w_j^\T (p+\norm{u}w_j) < \gamma_j \leqslant w_j^\T
      (p+u)}$. Therefore $w_j^\T u > \norm{u}$, which is a
    contradiction to the Cauchy-Schwarz inequality. Hence, we have
    \[
      k_p \leqslant \frac{1}{w_j^\T c} \leqslant \frac{1}{\min_{i \in I}
        w_i^\T c} \leqslant \frac{1}{\min_{w \in C^+, \norm{w}=1} w^\T
        c}
    \]
    which completes the proof.
  \end{proof}
  Note that the closedness of $C$ can be omitted if the inequality in
  the statement is turned strict. We use Definition \ref{def:3} in
  this article, because it has the advantage of being independent of
  any directions.
\end{remark}
Assumptions \ref{ass:A1}, \ref{ass:A2}, \ref{ass:A4}, and \ref{ass:A5}
imply that \eqref{P} is bounded: By \cite[][Theorem 7.1]{Roc70} the
sets $\{x \in X \mid g_i(x) \leqslant 0\}$ are closed for all
$i=1,\dots,m$ by lower semi-continuity. Therefore
$S = \bigcap_{i=1}^m \{x \in X\mid g_i(x) \leqslant 0\}$ is closed and
compact by \ref{ass:A4}. Now, since~$F$ is continuous by \ref{ass:A1},
$F[S]$ is compact as well. Finally,
because~${\interior C \neq\emptyset}$, there is some $y \in \R^q$ such
that $\P\subseteq\{y\}+C$. Moreover, Assumption \ref{ass:A2} implies
that \cite[see][Corollary 7.6.1]{Roc70}
${\relint \{x \in X \mid g_i(x) \leqslant 0\} = \{x \in X \mid g_i(x)
  < 0\}}$ for $i=1,\dots,m$ and Assumption \ref{ass:A3} implies that
\cite[see][Theorem 6.5]{Roc70}
\[
  \bigcap_{i=1}^m \relint\{x \in X\mid g_i(x) \leqslant 0\} = \relint
\bigcap_{i=1}^m\{x\in X\mid g_i(x)\leqslant 0\}.
\]
Therefore it holds
\begin{equation}\label{relintS}
  \relint S = \{x\in X\mid g(x) < 0\}
\end{equation}
and the set is nonempty.

For some parameter $w \in \R^q$ the problem
\begin{equation}\label{P1}
  \begin{aligned}
    \min \quad & w^\T F(x)\\
    \text{s.t.} \quad & g(x) \leqslant 0
  \end{aligned}
  \tag{P$_1$($w$)}
\end{equation}
is the well-known \emph{weighted sum scalarization of} (\ref{P}). By
Assumption \ref{ass:A1} and compactness of $S$ an optimal solution of
(\ref{P1}) exists for every $w \in \R^q$. The following is a common
result, see e.g. \cite{Jah84,Luc87}.
\begin{proposition}\label{prop:p1}
  Let $w \in C^+ \setminus \{0\}$. An optimal solution $x^w$ of
  \emph{(\ref{P1})} is a weak minimizer of \emph{(\ref{P})}.
\end{proposition}
We consider another scalarization \cite[see e.g.][]{Loe14,
  Ham14} that can be stated as
\begin{equation}\label{P2}
  \begin{aligned}
    \min \quad & z \\
    \text{s.t.} \quad & g(x) \leqslant 0, \\
    & Z^\T(F(x)-v-zc) \leqslant 0,
  \end{aligned}
  \tag{P$_2$($v,c$)}
\end{equation}
with a parameter $v \in \R^q$, that does typically not belong to $\P$,
and a direction $c \in \R^q$. The Lagrangian dual problem of
\eqref{P2} is given as
\begin{equation*}\label{D2}
  \begin{aligned}
    \max \quad & \inf_{x \in X \cap \dom g} \left\lbrace u^\T g(x) + w^\T F(x)
    \right\rbrace - w^\T v \\
    \text{s.t.} \quad & u \geqslant 0, \\
    & w^\T c = 1, \\
    & w \in C^+.
  \end{aligned}
  \tag{D$_2$($v,c$)}
\end{equation*}
The following primal-dual relationship between \eqref{P2} and
\eqref{D2} has been established in \cite[][Proposition
4.4]{Loe14} in a similar form. The proof is presented here due to
a flaw in the original work claiming that the feasible region of
\eqref{P2} is compact.
\begin{proposition}\label{prop:3.7}
  Let Assumptions {\normalfont\ref{ass:A1} -- \ref{ass:A5}} hold and
  let $p \in \interior\P$. Then for every~${v \in \R^q\setminus\P}$
  and $c := p-v$, solutions $(x^*,z^*)$ and $(u^*,w^*)$ to \eqref{P2}
  and~\eqref{D2}, respectively, exist and their optimal values
  coincide.
\end{proposition}
\begin{proof}
  By \cite[][Corollary 6.6.2]{Roc70} we have
  $\interior \P = \relint F[S] + \interior C$. Assumption \ref{ass:A1}
  and \cite[Theorem 6.6]{Roc70} yield that
  $\relint F[S] \subseteq F[\relint S]$. Therefore we can write
  $p \in \interior \P$ as~${p = F(x) + \bar{c}}$ for some
  $x \in \relint S$ and $\bar{c} \in \interior C$. From Assumption
  \ref{ass:A5} we conclude
  \begin{equation}\label{prop3.6_*}
    Z^\T(F(x)-v-c) = Z^\T(F(x)-p) = -Z^\T\bar{c} < 0.
    \tag{$*$}
  \end{equation}
  This implies that $(x,1)$ is feasible for \eqref{P2}. Since
  $v \notin \P$, the second constraint of~\eqref{P2} is violated
  whenever $z \leqslant 0$. From Assumptions \ref{ass:A1},
  \ref{ass:A2}, and \ref{ass:A4} it follows that the set
  \[
    \{(x,z) \in \R^{n+1} \mid g(x) \leqslant 0, Z^\T(F(x)-v-zc)
    \leqslant 0, z \leqslant 1\}
  \]
  is compact and nonempty. Thus there exists an optimal solution
  $(x^*,z^*)$ of \eqref{P2} by the extreme value theorem and one has
  $0 \leqslant z^* \leqslant 1$. Next, observe that $(x,1)$ is also
  strictly feasible for \eqref{P2} by Equations \eqref{relintS} and
  \eqref{prop3.6_*}. This is the well-known Slater's constraint
  qualification. Consequently strong duality holds, i.e. there exists
  an optimal solution $(u^*,w^*)$ of \eqref{D2} and the optimal values
  coincide.
\end{proof}
Similar to Proposition \ref{prop:p1} we obtain weak minimizers of
\eqref{P} from solutions of~\eqref{P2}. The following is Proposition
4.5 from \cite{Loe14}.
\begin{proposition}\label{prop:p2wmin}
  Let $(x^*,z^*)$ be a solution to \eqref{P2}. Then $x^*$ is a weak
  minimizer of \eqref{P} and $y:=v+z^*c$ is a weakly $C$-minimal
  element of $\P$.
\end{proposition}

\section{An Algorithm for Bounded VCPs with Vertex
  Selection}\label{sec:4}
In this section we present an algorithm for computing a weak
$\varepsilon$-solution for Problem~\eqref{P}.  The algorithm computes
a shrinking sequence~$(\O^k)$ of polyhedral outer approximations and a
growing sequence $(\I^k)$ of polyhedral inner approximations of the
upper image $\P$, i.e. one has
\begin{equation}
  \O^0 \supseteq \O^1 \supseteq \dots \supseteq \P \supseteq \dots
  \supseteq \I^1 \supseteq \I^0.
\end{equation}
This is achieved by iteratively cutting off vertices $v$ of $\O^k$
while introducing new halfspaces. The algorithm is a modification of
the primal approximation algorithm presented in \cite{Loe14}. The
difference lies in the way the approximations are updated. While in
\cite{Loe14} there is no rule stated how to choose the next
vertex, we employ a vertex selection that takes into account
$\haus{\O^k}{\I^k}$. Therefore $\haus{\O^k}{\I^k}$ is computed in each
iteration by solving certain convex quadratic subproblems. We
formulate Corollary~\ref{cor:4.4} to show that the vertex selection
can be performed efficiently. 
The algorithm consists of two parts, an
initialization phase and an update phase, which we will explain in
detail below. Correctness is shown in Theorem \ref{thm:4.3}.

\begin{init}
  In the initialization phase an initial outer approximation $\O^0$
  and an initial inner approximation $\I^0$ of $\P$ are computed. To
  obtain $\O^0$, (P$_1$($z^j$)) is solved for every column $z^j$ of
  $Z$. Solutions $x^j$ are weak minimizers of \eqref{P} according to
  Proposition~\ref{prop:p1} and give rise to the following hyperplanes
  that support $\P$ at $F(x^j)$:
  \begin{equation}
    \mathcal{H}_j := \{y \in \R^q \mid z^{j\T}y = z^{j\T}F(x^j)\}.
  \end{equation}
  Thus, we can define $\O^0$ as the intersection of all halfspaces
  $\mathcal{H}_j^+$ that are defined by~$\mathcal{H}_j$, i.e.
  \begin{equation}\label{O0}
    \O^0 := \bigcap_{j=1}^{\ell} \mathcal{H}_j^+ = \bigcap_{j=1}^{\ell} \{y \in
    \R^q \mid z^{j\T}y \geqslant z^{j\T}F(x^j)\}.
  \end{equation}
  Note that $\O^0$ has at least one vertex, because \eqref{P} is
  bounded and $C$ is an ordering cone, in particular pointed.  An
  initial inner approximation $\I^0$ is readily available at no
  additional cost by setting
  \begin{equation}\label{I0}
    \I^0 := \conv \{F(x^j) \mid j = 1,\dots,\ell\} + C.
  \end{equation}
\end{init}
\begin{mainloop}
  During the update phase the current approximations are refined. In
  order to do so, supporting hyperplanes to the upper image are
  computed from solutions of \eqref{P2} and \eqref{D2} according to
  the following proposition \cite[see][Proposition 4.7]{Loe14}.
  \begin{proposition}\label{prop:hyper}
    Let $(x^*,z^*)$ and $(u^*,w^*)$ be solutions of \eqref{P2} and
    \eqref{D2}, respectively. Then the hyperplane
    \[
      \mathcal{H} := \{y \in \R^q \mid w^{*\T}y = w^{*\T}v+z^*\}
    \]
    is a supporting hyperplane of $\P$ at $y^*:=v+z^*c$.
  \end{proposition}
  In iteration $k$ the input parameters for \ref{P2} are chosen by
  means of the following vertex selection procedure (VS).
  \begin{vs}
    For every $s \in \vertices \O^k$ the euclidean distance to $\I^k$
    is computed by solving
    \begin{equation*}\label{QP}
      \tag{QP($s,\I^k$)}
      \begin{aligned}
        \min \quad & \norm{p-s}^2 \\
        \text{s.t.} \quad & p \in \I^k.
      \end{aligned}
    \end{equation*}
    Note that \eqref{QP} lives in the image space of \eqref{P} and is
    convex quadratic. Next we consider the following bilevel
    optimization problem
    \begin{equation*}\label{VS}
      \tag{VS($\O^k,\I^k$)}
      \begin{aligned}
        \max \quad & \norm{p^*-s} \\
        \text{s.t.} \quad & s \in \vertices \O^k \\
        & p^* \; \text{solves} \; \eqref{QP}.
      \end{aligned}
    \end{equation*}
    A solution to \eqref{VS} is a vertex of $\O^k$ that yields the
    shortest distance to the current inner approximation. Since
    $\O^k \supseteq \I^k$ by construction, we obtain the Hausdorff
    distance $\haus{\O^k}{\I^k}$ easily from a solution of \eqref{VS}
    as explained in the next corollary.
    \begin{corollary}\label{cor:haus}
      Let $\O,\I \subseteq \R^q$ be polyhedra with the same pointed
      recession cone and~${\O \supseteq \I}$. Further let $(s^*,p^*)$
      be a solution of \emph{(VS($\O$,$\I$))}. Then
      \[
        \haus{\O}{\I} = \norm{p^*-s^*}.
      \]  
    \end{corollary}
    \begin{proof}
      As $\O \supseteq \I$ and by Equation \eqref{eq:4}, the maximum
      in the definition of $d_{\mathsf{H}}$ is attained as
      \[
        \adjustlimits\max_{s \in \vertices \O} \min_{p \in \I} \norm{p-s}.
      \]
      Since squaring the norm in the objective function of (QP) does
      not change the solution, we get
      \begin{align*}
        \haus{\O}{\I} &= \adjustlimits\max_{s \in \vertices \O} \min_{p \in \I}
                        \norm{p-s} \\
                      &= \max_{s \in \vertices \O} \left\lbrace \norm{p^*-s} \mid p^* \;
                        \text{is a solution of (QP($s,\I$))} \right\rbrace \\
                      &= \norm{p^*-s^*}. \qedhere
      \end{align*}  
    \end{proof}
  \end{vs}
  Note that solving \eqref{VS} amounts to solving \eqref{QP} for every
  vertex $s$ of~$\O^k$ and taking a maximum over a finite set. If
  $\haus{\O^k}{\I^k} \leqslant \varepsilon$, then
  $\haus{\O^k}{\P} \leqslant \varepsilon$ and
  $\haus{\P}{\I^k} \leqslant \varepsilon$ follow immediately from the
  fact that $\O^k \supseteq \P \supseteq \I^k$. In this case a weak
  $\varepsilon$-solution $\mathcal{X}$ to \eqref{P} is returned.
  Otherwise we set $v := s^*$ and~${c := p^*-s^*}$ and solve
  \ref{P2}. Thereby we obtain a supporting hyperplane $\mathcal{H}$
  of~ $\P$ according to Proposition \ref{prop:hyper} and set
  \begin{equation}\label{eq:15}
    \begin{aligned}
      \O^{k+1} &= \O^k \cap \mathcal{H}^+, \\
      \I^{k+1} &= \cl\conv \left(\I^k \cup \{F(x^*)\}\right),
    \end{aligned}
  \end{equation}
  where $x^*$ solves \eqref{P2}. Also, $x^*$ is appended to the
  solution set $\mathcal{X}$. Note, that the closure in Equation
  \eqref{eq:15} is necessary, because we are dealing with unbounded
  sets. However, $\I^{k+1}$ does not have to be computed explicitly as
  we are only interested in its vertices.  Pseudocode is presented in
  Algorithm \ref{alg:1} and one iteration of the algorithm is
  illustrated in Figure \ref{fig:it}.
\end{mainloop}
\begin{algorithm}
  \DontPrintSemicolon
  \KwData{Problem \eqref{P}, accuracy $\varepsilon > 0$, max. no. of
    iterations $K$}
  \KwResult{Weak $\varepsilon$-solution $\mathcal{X}$ of \eqref{P},
    vertices $\O$/$\I$ of an outer/inner approximation of $\P$
    \textbf{or} max. no. of iterations exceeded}
  Compute a solution $x^j$ to (P$_1$($z^j$)) for $j=1,\dots,\ell$\;
  $\mathcal{X} \isassigned \{x^j \mid j=1,\dots,\ell\}$\;
  Compute an outer approximation $\O^0$ according to \eqref{O0}\;
  Compute an inner approximation $\I^0$ according to \eqref{I0}\;
  $k \isassigned 0$, $d_{\mathsf{H}} \isassigned \infty$\;
  \Repeat{$d_{\mathsf{H}} \leqslant \varepsilon$ \textnormal{or} $k=K$}{
    Compute a solution $(s,p)$ to \eqref{VS}\;
    $d_{\mathsf{H}} \isassigned \norm{p-s}$\;
    \If{$d_{\mathsf{H}} > \varepsilon$}{
      $v \isassigned s$, $c \isassigned p-s$\;
      Compute solutions $(x,z)$/$(u,w)$ to \eqref{P2}/\eqref{D2}\;
      $\mathcal{X} \isassigned \mathcal{X} \cup \{x\}$\;
      $\O^{k+1} \isassigned \O^k \cap \{y \in \R^q \mid w^\T y
      \geqslant w^\T v+z\}$\;
      $\I^{k+1} \isassigned \cl\conv(\I^k \cup \{F(x)\})$\;
      $k \isassigned k+1$\;
    }
  }
  $\O \isassigned \vertices \O^k$\;
  $\I \isassigned \vertices \I^k$\;
  \KwRet{$\mathcal{X}$, $\O$, $\I$}
  \caption{A Benson-type Algorithm with Vertex Selection for \eqref{P}}\label{alg:1}
\end{algorithm}
\begin{figure}
  \centering
  \includegraphics{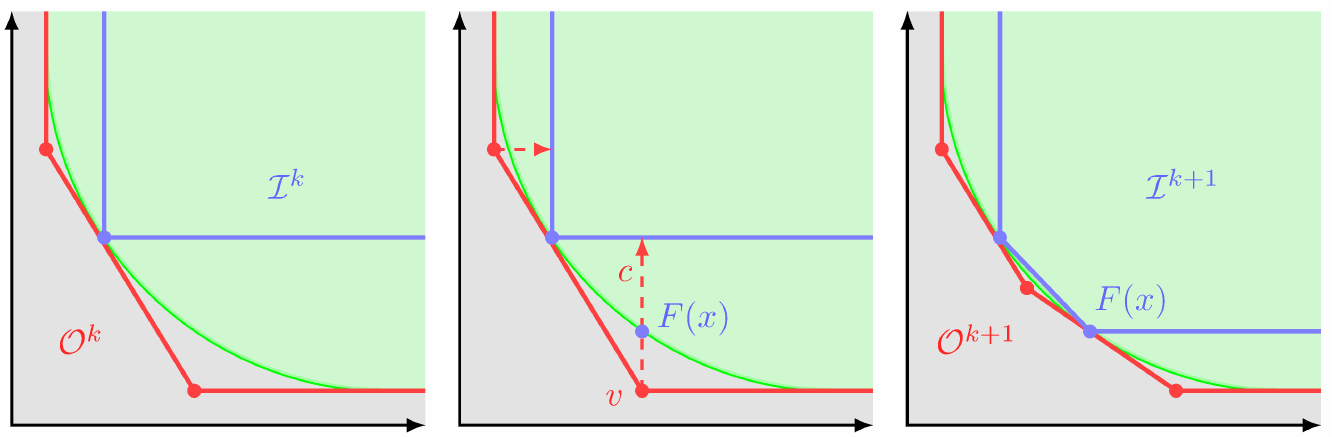}
  \decoRule
  \caption{\label{fig:it}%
    \textbf{Left:} Outer (red) and inner (blue) approximations of
    $\mathcal{P}$ after iteration~$k$. \textbf{Center:} The vertex $v$
    and direction $c$ are obtained by the vertex
    selection. 
    The point~$F(x)$ is obtained by solving
    \eqref{P2}. \textbf{Right:} The updated outer and inner
    approximations after cutting off $v$ and adding $F(x)$ as a vertex
    to $\mathcal{I}^k$.}
\end{figure}
\begin{theorem}\label{thm:4.3}
  Under Assumptions {\normalfont\ref{ass:A1} -- \ref{ass:A5}}
  Algorithm \ref{alg:1} is correct, i.e. if it terminates with $k<K$
  it returns a weak $\varepsilon$-solution of \eqref{P}.
\end{theorem}
\begin{proof}
  Optimal solutions to (P$_1$($z^j$)) exist for all $j=1,\dots,\ell$
  by Assumptions~\ref{ass:A1},~\ref{ass:A2}, and
  \ref{ass:A4}. Therefore line 1 is valid and the set $\mathcal{X}$
  initialized in line 2 is nonempty. Proposition \ref{prop:p1} states
  that $\mathcal{X}$ only contains weak minimizers of \eqref{P} and
  implies that the vertices of $\I^0$ are weakly $C$-minimal elements
  of $\P$. Because $C$ is a pointed cone, the set~$\O^0$ has at least
  one vertex. Therefore the problem (VS($\O^0,\I^0$)) has a
  solution. Optimal solutions to \eqref{P2} and \eqref{D2} exist
  according to Proposition \ref{prop:3.7}. By Proposition
  \ref{prop:p2wmin} a weak minimizer of \eqref{P} is added to
  $\mathcal{X}$ in line 12 and~$\I^k$ is updated with a new vertex
  that is weakly $C$-minimal in $\P$. Now, $\I^{k+1} \subseteq \P$,
  because it is the generalized convex hull of finitely many weakly
  $C$-minimal points and directions of $\P$. Moreover, as the
  hyperplane $\{y \in \R^q \mid w^\T y=w^\T v+z\}$ supports $\P$ in
  $v+zc$, the set~$\O^{k+1}$ in line 13 is nonempty, has a vertex, and
  satisfies $\O^{k+1} \supseteq \P$. Note that by Corollary
  \ref{cor:haus}, $d_{\mathsf{H}}$ defined in line 8 is the Hausdorff
  distance between the current approximations $\O^k$ and
  $\I^k$. Therefore, assuming termination with $k < K$, the algorithm
  terminates if the Hausdorff distance between the current outer and
  inner approximation of $\P$ is less than or equal to the error
  margin $\varepsilon$. Assume this is the case after $\kappa$
  iterations. We must show that $\mathcal{X}$ is a weak
  $\varepsilon$-solution of \eqref{P}. Clearly, $\mathcal{X}$ is
  finite and, by Propositions \ref{prop:p1} and \ref{prop:p2wmin}, it
  consists of weak minimizers only. Moreover we
  have~${\haus{\O^\kappa}{\I^\kappa} \leqslant \varepsilon}$ and
  therefore $\haus{\P}{\I^\kappa} \leqslant \varepsilon$. Finally, due
  to its construction,~$\I^\kappa$ can be written
  as~${\I^\kappa = \conv F[\mathcal{X}]+C}$. Hence, $\mathcal{X}$
  fulfills the definition of a weak $\varepsilon$-solution which
  completes the proof.
\end{proof}
\begin{vs_efficient}
  So far, the main drawback of the vertex selection is that it
  requires \eqref{QP} to be solved for every $s \in\vertices\O^k$. In
  order to make VS efficient, we make the following observation about
  the input parameters: From one iteration to the next, the inner
  approximation only changes by introducing one new vertex. Therefore
  the solutions of \eqref{QP} and~ (QP($s,\I^{k+1}$)) may be
  identical. We can exploit this structure by checking a single
  inequality to determine whether, for a given vertex $s$ of $\O^k$,
  we have to solve \eqref{QP}. The following result captures this
  idea.
  \begin{corollary}\label{cor:4.4}
    Let the iteration be $k+1$ in Algorithm \ref{alg:1}. Let $s$ be a
    vertex of both $\O^k$ and $\O^{k+1}$, let $p^*$ be a solution to
    \eqref{QP} and $F(x)$ such
    that~${\I^{k+1} = \cl\conv(\I^k \cup \{F(x)\})}$. Then the
    following are equivalent:
    \begin{enumerate}[label=(\roman*)]
    \item $p^*$ is a solution to \emph{(QP($s,\I^{k+1}$))},\label{cor4.4_statement_i}
    \item $(p^*-s)^\T(F(x)-p^*) \geqslant 0$.\label{cor4.4_statement_ii}
    \end{enumerate}
  \end{corollary}
  \begin{proof}
    This is a straightforward consequence of convexity and a standard
    result in convex optimization. Given a convex optimization problem
    with differentiable objective function $f$ and feasible region $S$
    the following are equivalent, see \cite[][Section
    4.2.3.]{Boy04}:
    \begin{enumerate}[label=(\alph*)]
    \item $p^* \in S$ is a solution,\label{cor4.4_proof_i}
    \item $\nabla f(p^*)^\T(p-p^*) \geqslant 0$ for all $p \in S$.\label{cor4.4_proof_ii}
    \end{enumerate}
    Together with $\nabla \left[ \norm{p^*-s}^2 \right] = 2(p^*-s)$,
    \ref{cor4.4_statement_i} is equivalent to
    \begin{equation*}
      (p^*-s)^\T (p-p^*) \geqslant 0 \; \text{for all} \; p \in \I^{k+1}.
    \end{equation*}
    This inequality holds in particular for $p=F(x) \in
    \I^{k+1}$. Therefore \ref{cor4.4_statement_i} implies
    \ref{cor4.4_statement_ii}. On the other hand, assume that
    \ref{cor4.4_statement_ii} holds and $p^*$ is not a solution to
    (QP($s,\I^{k+1}$)). Then, as $p^*$ solves \eqref{QP}, there must
    exist some $\bar{p} \in \I^{k+1}$, such that
    \[
      (p^*-s)^\T(\bar{p}-p^*) < 0.
    \]
    By the definition of $\I^{k+1}$, $\bar{p}$ can be written as
    $\bar{p} = \lambda F(x) + (1-\lambda)y + c$ for some
    $0 \leqslant \lambda \leqslant 1$, $y \in \I^k$, and $c \in
    C$. Altogether this yields
    \begin{align*}
      0 &> (p^*-s)^\T(\bar{p}-p^*) \\
        &= \lambda\underbrace{(p^*-s)^\T(F(x)-p^*)}_{\substack{\geqslant\;
          0,\;\text{bc. $p^*$ is optimal}\\\text{for
      (QP($s,\I^k$))}}}+(1-\lambda)\underbrace{(p^*-s)^\T(y-p^*)}_{\geqslant\;
      0\;\text{by \ref{cor4.4_proof_ii} for}\;S=\I^k}+(p^*-s)^\T c\\
        &\geqslant (p^*-s)^\T(\underbrace{p^*+c}_{\in \I^k}-p^*)\\
        &\geqslant 0.
    \end{align*}
    This is a contradiction. Thus $p^*$ solves (QP($s,\I^{k+1}$)) and
    the proof is complete.
  \end{proof}
\end{vs_efficient}

\section{Numerical Examples}\label{sec:5}
In this section we present three examples and compare computational
results with the primal algorithm in \cite{Loe14} illustrating
the benefits of the vertex selection approach.  Moreover we present an
application of Algorithm \ref{alg:1} to the problem of regularization
parameter tracking in machine learning as suggested in
\cite{Gie19,Nus19}, as well as an example from structural mechanics
with non-differentiable objective functions.  The algorithms are
implemented in MATLAB R2016b. Solving the scalar optimization problems
is done with \emph{CVX} v2.1, a package for specifying and solving
convex programs \cite{Gra14,Boy08}, and \emph{GUROBI}~v8.1
\cite{Gur19}. We use \emph{bensolve tools}
\cite{Wei16,Cir18}, a toolbox for polyhedral calculus and
polyhedral optimization, to handle the outer and inner approximations
of the upper image, in particular to compute a $V$-representation of
the outer approximation in every iteration. All experiments are
conducted on a machine with a 2.2GHz Intel Core i7 and~8GB RAM.

\begin{example}\label{ex:1}
  We consider an academic example where the feasible region is an
  axially parallel ellipsoidal body with semi-axes of lengths 1, $a$,
  and 5. Here $a \in \R_{++}$ is any parameter. Thus, by variying $a$
  we can steer how dilated the body is along the~$x_2$-axis.
  Altogether the problem can be formulated as
  \begin{align*}
    &\min \; F(x) = \begin{pmatrix} x_1\\x_2\\x_3 \end{pmatrix} \;\text{w.r.t.} \; \leqslant \\
    &\;\text{s.t.} \; \left(\frac{x_1-1}{1}\right)^2+\left(\frac{x_2-1}{a}\right)^2+\left(\frac{x_3-1}{5}\right)^2
      \leqslant 1.
  \end{align*}
  Computational data can be seen in Table \ref{tbl:1} for
  $\varepsilon =0.05$ and different values of $a$.  It shows that the
  performance of the algorithm with VS is not affected by the choice
  of $a$. However, without VS the number of scalarizations to solve
  scales with the magnitude of $a$.  This also has a notable impact on
  the computation time. Moreover the algorithm with VS computes
  approximately half as many minimizers, thus obtaining a coarser
  approximation.  These effects can be observed in Figure
  \ref{fig:ex1} which displays the inner approximations computed by
  both algorithms for $a=7$. Table \ref{tbl:2} shows the impact of
  Corollary \ref{cor:4.4}. On average 82\% of the quadratic
  subproblems can be spared, making VS very efficient.
  \begin{table}
    \centering
    \caption{Experimental data for Example \ref{ex:1} for
      $\varepsilon=0.05$. It displays the computation time with and
      without VS as well as the size $\lvert \mathcal{X} \rvert$ of
      the solution set for different values of $a$.}\label{tbl:1}
    \begin{tabular}{crc}
      \toprule
      \boldmath{$a$} & \cell{cc}{\multicolumn{2}{c}{\textbf{time}} \\
      \textbf{VS} \cmark & \textbf{VS} \xmark} &
                                                 \cell{ll}{\multicolumn{2}{c}{\boldmath{$\lvert \mathcal{X}
                                                 \rvert$}} \\ \textbf{VS} \cmark & \textbf{VS} \xmark}
      \\
      \midrule
      5 & \cell{rr}{39.72 & 129.03} & \cell{lr}{66\;\; & \;\;114} \\
      7 & \cell{rr}{42.67 & 147.42} & \cell{lr}{72\;\; & \;\;127} \\
      10 & \cell{rr}{47.21 & 175.51} & \cell{lr}{78\;\; & \;\;139} \\
      20 & \cell{rr}{45.15 & 186.65} & \cell{lr}{76\;\; & \;\;154} \\
      \bottomrule
    \end{tabular}
  \end{table}
  \begin{figure}
    \centering
    \subfloat[{with \vrtxslct}]{%
      \label{fig:ex1vs}%
      \includegraphics[width=.3\textwidth]{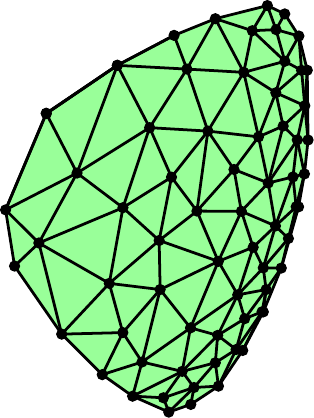}%
    }
    \hfil
    \subfloat[{without \vrtxslct}]{%
      \label{fig:ex1wo}\includegraphics[width=.3\textwidth]{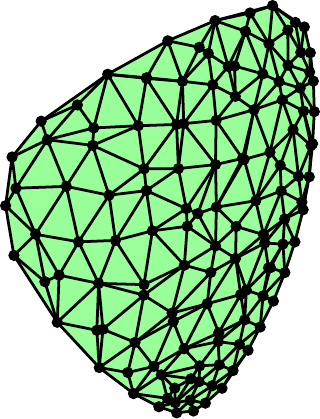}}

    \vspace{10pt}
    \decoRule
    \caption{\label{fig:ex1}
      Inner approximations of the upper image for Example~\ref{ex:1}
      with $a = 7$; with \vrtxslct~\protect\subref{fig:ex1vs} and
      without~\protect\subref{fig:ex1wo}.  Each
      vertex corresponds to a weak minimizer. One can see that without
      \vrtxslct\ there are many vertices in close proximity to
      each other, especially in regions that exhibit a large curvature
      (bottom right). With \vrtxslct, the vertices are ``spread
      more evenly'' across the surface.%
    }
  \end{figure}
  \begin{table}
    \centering
    \caption{Number of quadratic problems solved with and without
      using the equivalence in Corollary \ref{cor:4.4} for different
      values of $a$ and $\varepsilon=0.05$ in Example
      \ref{ex:1}.}\label{tbl:2}
    \begin{tabular}{cr}
      \toprule
      \boldmath{$a$} &
                       \cell{rr}{\multicolumn{2}{c}{\textbf{Cor. \ref{cor:4.4}}}
      \\
      \cmark & \xmark} \\
      \midrule
      5 & \cell{rr}{417 & 3348} \\
      7 & \cell{rr}{640 & 4060} \\
      10 & \cell{rr}{1019 & 4849} \\
      20 & \cell{rr}{871 & 4420} \\
      \bottomrule
    \end{tabular}
  \end{table}
\end{example}
\begin{example}[Regularization parameter tracking in machine learning]\label{ex:2}
  Regularized learning has been a common practice in machine learning
  over the past years. One of the heavily studied approaches is the
  \emph{elastic net}:
  \begin{equation}\label{enet}
    \min \; \alpha_1\norm{Ax-b}^2 + \alpha_2\norm{x}_1 +
    \alpha_3\norm{x}^2,
  \end{equation}
  where $A$ and $b$ are a matrix and a vector of appropriate sizes
  containing observed data and $\norm{\cdot}_1$ denotes the
  $\ell_1$-norm. The weight vector
  $\alpha = (\alpha_1,\alpha_2,\alpha_3)^\T$ steers the influence of
  the loss function $\norm{Ax-b}^2$ and the regularization terms
  $\norm{x}_1$ and $\norm{x}^2$ relative to each other.  The task of
  choosing $\alpha$ is called \emph{regularization parameter tracking}
  and is a difficult problem on its own. While there are approaches to
  this problem for certain classes \cite[see][]{Fis15,Efr04}, often
  one has to solve Problem \eqref{enet} for every $\alpha$ on a grid
  in the parameter domain.  The authors of \cite{Gie19} propose a new
  method by observing that Problem \eqref{enet} is the weighted sum
  scalarization of the VCP
  \begin{equation}\label{en:CVOP}
    \min \begin{pmatrix} \norm{Ax-b}^2 \\ \norm{x}_1 \\    
      \norm{x}^2 \end{pmatrix} \; \text{w.r.t.} \leqslant.
  \end{equation}
  Applying Algorithm \ref{alg:1} to that problem yields a weak
  $\varepsilon$-solution $\mathcal{X}$ in which each weak minimizer
  corresponds to a different choice of $\alpha$. By the definition of
  an infimizer we have that for every $\alpha \in \R^3$ there is some
  $x \in \mathcal{X}$ which is $\varepsilon$-optimal for Problem
  \eqref{enet}. Therefore we obtain a selection of parameters that is
  optimal up to a tolerance of $\varepsilon$.

  The elastic net is frequently used in microarray classification and
  gene selection, a problem in computational biology.  A key
  characteristic of such problems is that the dimension of the
  variable space is much larger than the number of observations. As
  overfitting is a major concern in such a scenario, regularized
  approaches are favorable \cite[cf.][]{Has05}. Due to the problem
  dimension, solving scalarizations becomes costly. Therefore VS may
  be advantageous whenever $n \gg q$. We applied the elastic net to
  the following data sets:
  \begin{itemize}
    \setlength\itemsep{0em}
  \item {\texttt{Lung}} \cite{Mra07} with $n=$\,12,600 features and $m=203$ instances,
  \item {\texttt{arcene}} \cite{Guy05} with $n=$\,10,000 and $m=100$,
  \item {\texttt{GLI-85}} \cite{Zha10} with $n=$\,22,283 and $m=85$,
  \item {\texttt{MLL}} \cite{Mra07} with $n=$\,12,582 and $m=72$,
  \item {\texttt{Ovarian}} \cite{Pet02} with $n=$\,15,154 and $m=253$,
  \item {\texttt{SMK-CAN-187}} \cite{Zha10} with $n=$\,19,993 and $m=187$,
  \item {\texttt{14-cancer}} \cite{Has09} with $n=$\,16,063 and $m=198$.
  \end{itemize}
  The data sets have been scaled such that the response is centered and the
  predictors are standardized:
  \[
    \sum_{i=1}^m b_i = 0, \quad \sum_{i=1}^m A_{i,j} = 0, \quad
    \sum_{i=1}^m A_{i,j}^2 = 1,
  \]
  for $j=1,\dots,n$. We use 70\% of the data for training and 30\%
  for testing. Table \ref{tbl:3} shows the approximation errors and
  the test data mean squared error (MSE) after one hour of
  runtime. Evidently the approximation error is smaller with vertex
  selection in all test cases, while the MSE is mostly unaffected by
  the chosen method.
  \begin{table}
    \centering
    \caption{Experimental data for Example \ref{ex:2}. Highlighted in green are the lower ones of the MSEs computed by the methods for every data set.}\label{tbl:3}
    \begin{tabular}{l*{4}{c}}
      \toprule
      \textbf{Data Set} & \textbf{VS} & \boldmath{$\varepsilon$} & \textbf{MSE}\\
      \midrule
      \cell{l}{{\texttt{Lung}}} & \cell{l}{\cmark\\\xmark} & \cell{r}{0.3017\\0.8549} & \cell{r}{0.3685\\\textcolor{ForestGreen}{0.3353}}\\
      \cell{l}{{\texttt{arcene}}} & \cell{l}{\cmark\\\xmark}  & \cell{r}{0.0179\\0.0296} & \cell{r}{0.1799\\\textcolor{ForestGreen}{0.1682}}\\
      \cell{l}{{\texttt{GLI-85}}} & \cell{l}{\cmark\\\xmark} & \cell{r}{0.0251\\0.0330} & \cell{r}{0.0915\\\textcolor{ForestGreen}{0.0823}}\\
      \cell{l}{{\texttt{MLL}}} & \cell{l}{\cmark\\\xmark} & \cell{r}{0.0161\\0.0340} & \cell{r}{\textcolor{ForestGreen}{0.0546}\\\textcolor{ForestGreen}{0.0546}}\\
      \cell{l}{{\texttt{Ovarian}}} & \cell{l}{\cmark\\\xmark} & \cell{r}{0.6380\\0.7442} & \cell{r}{\textcolor{ForestGreen}{0.0096}\\\textcolor{ForestGreen}{0.0096}}\\
      \cell{l}{{\texttt{SMK-CAN-187}}} & \cell{l}{\cmark\\\xmark} & \cell{r}{0.4126\\0.7774} & \cell{r}{\textcolor{ForestGreen}{0.1655}\\0.1659}\\
      \cell{l}{{\texttt{14-cancer}}} & \cell{l}{\cmark\\\xmark} & \cell{r}{4.4348\\11.0260} & \cell{r}{\textcolor{ForestGreen}{7.7194}\\7.7961}\\
      \bottomrule
    \end{tabular}
  \end{table}
\end{example}
\begin{example}[Planar truss design]\label{ex:3}
  In this example we discuss a problem from structural mechanics with
  non-differentiable objective function. We consider a planar truss
  that consists of two fixed supports and four free nodes which are
  connected by ten beams as depicted in Figure \ref{fig:truss}. The
  beams are assumed to have the same cross sectional area, density,
  and Young's modulus. Our aim is to distribute a net force $F$ among
  the four free nodes in such a way that the absolute displacement of
  each of these nodes is minimized. We set the following problem
  parameters:
  \begin{center}
    \begin{tabular}{lrl}
      \toprule
      beam length $\ell$ & $9000$ & mm \\
      beam radii & 25 & mm \\
      Young's modulus & 70,000 & N/mm$^2$ \\
      force $F$ & 150,000 & N \\
      \bottomrule
    \end{tabular}
  \end{center}
  For simplicity we assume a linear elasticity model. We have a total
  of eight variables, i.e. a horizontal and a vertical force in every
  free node, and four objectives, i.e. the maximum of the horizontal
  and vertical displacement of each free node. The relationship
  between the acting forces $p \in \R^8$ and the nodal displacements
  $d \in \R^8$ is given by
  \begin{equation}
    d = K^{-1} p,
  \end{equation}
  where $K \in \R^{8 \times 8}$ is called the structure stiffness
  matrix of the truss. $K$ depends on each beams length, radius, and
  rotation as well as the Young's modulus. For more insight from a
  mechanical viewpoint we refer the reader to the vast amount of
  literature on the design of trusses, such as
  \cite{Ben03,Rot17}. For the optimization we induce
  bounds on the tension and compression in each beam of 170
  N/mm$^2$. Altogether the problem can be posed as
  \begin{align*}
    &\min \begin{pmatrix} \max \left\{ \left\lvert d_{1,h}
        \right\rvert, \left\lvert d_{1,v} \right\rvert \right\} \\
      \vdots \\ \max \left\{ \left\lvert d_{4,h}
        \right\rvert, \left\lvert d_{4,v} \right\rvert \right\}\end{pmatrix} \; \text{w.r.t.} \; \leqslant
    \\
    &\; \text{s.t.} \; \left\{\begin{aligned}
        d &= K^{-1} p\\
        d &= (d_{i,h},d_{i,v})_{i=1,\dots,4}^\T\\
        e^\T p &= F \\
        -170 &\leqslant Td \leqslant 170
      \end{aligned}\right.
  \end{align*}
  where $d_{i,h}$, $d_{i,v} \in \R$ denote the horizontal and vertical
  displacements of node $i$, respectively, $e \in \R^{8}$ is the
  vector of all ones, and $T \in \R^{10 \times 8}$ is a matrix
  relating the nodal displacements to the stress in the beams. Note
  that the problem can also be formulated as a vector linear
  program. The computational results are reported in Table
  \ref{tbl:4}. As in the previous examples, a smaller solution set is
  computed with VS. In a practical sense this eases a decision makers
  choice, particularly because individual minimizers may be very
  different from each other, see Figure \ref{fig:truss}.
  \begin{figure}
    \centering
    \includegraphics{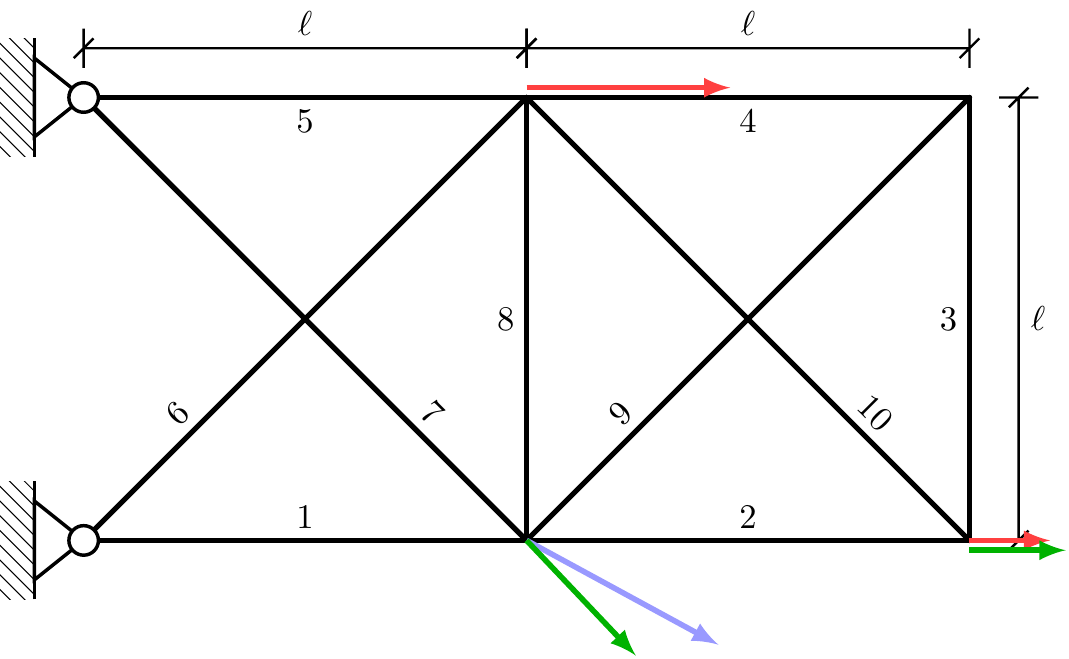}    
    \decoRule
    \caption{\label{fig:truss}%
      The planar 10-member truss from Example \ref{ex:3} with two
      fixed supports and four free nodes. The colored arrows
      illustrate different loads corresponding to weak minimizers.
    }
  \end{figure}
  \begin{table}
    \centering
    \caption{Results for Example \ref{ex:3}.}\label{tbl:4}
    \begin{tabular}{crc}
      \toprule
      \boldmath{$\varepsilon$} & \cell{cc}{\multicolumn{2}{c}{\textbf{time}} \\
      \textbf{VS} \cmark & \textbf{VS} \xmark} &
                                                 \cell{ll}{\multicolumn{2}{c}{\boldmath{$\lvert \mathcal{X}
                                                 \rvert$}} \\ \textbf{VS} \cmark & \textbf{VS} \xmark}
      \\
      \midrule
      0.5 & \cell{rr}{10.51 & 31.52} & \cell{lr}{25\;\;\; & \;\;\;88} \\
      0.4 & \cell{rr}{12.28 & 35.00} & \cell{lr}{28\;\;\; & \;\;\;95} \\
      0.3 & \cell{rr}{12.80 & 37.13} & \cell{lr}{29\;\; & \;\;103} \\
      0.2 & \cell{rr}{20.71 & 40.83} & \cell{lr}{43\;\; & \;\;110} \\
      \bottomrule
    \end{tabular}
  \end{table}  
\end{example}

\section{Conclusion}
We have proposed vertex selection, a new update rule for polyhedral
approximations in Benson-type algorithms for VCPs. We have shown that
VS can be performed efficiently. Moreover, the approximation error is
known in every iteration of the algorithm and in the provided examples
fewer scalarizations need to be solved. Hence one obtains coarser
solutions of VCPs with the same approximation quality while saving
computation time.

%
\bibliographystyle{bib/tfs}
\bibliography{bib/references}

\end{document}